\documentclass[preprint]{elsarticle}

\usepackage{amsfonts, amsthm, amsmath}
\usepackage{mathrsfs}

\usepackage[english]{babel}
\usepackage{graphicx}

\usepackage{babelbib}

\theoremstyle{plain}
  \newtheorem{theorem}{Theorem}
  \newtheorem{proposition}[theorem]{Proposition}
  \newtheorem{lemma}[theorem]{Lemma}
  
\theoremstyle{definition}
  \newtheorem{definition}[theorem]{Definition}
\theoremstyle{remark}
  \newtheorem{example}{Example}
  \newtheorem{remark}[theorem]{Remark}

\newcommand{\gra}[1]{\left\{ #1 \right\}}
\newcommand{\Prob}{\mathbb{P}}
\newcommand{\E}{\mathbb{E}}

\newcommand{\punto}{\textrm{.}}
\newcommand{\virgola}{\textrm{,}}
\newcommand{\ldue}{L^2\pa{0,T}}
\newcommand{\erre}{\mathbb{R}}
\newcommand{\reals}{\mathbb{R}}

\newcommand{\cD}{\mathbb{D}}
\newcommand{\cB}{\mathcal{B}}

\newcommand{\qua}[1]{\left[ #1 \right]}

\newcommand{\pa}[1]{\left( #1 \right)}

\newcommand{\abs}[1]{\left| #1 \right|}
\newcommand{\norm}[1]{\left\| #1 \right\|}
\newcommand{\ang}[1]{\left< #1 \right>}

\begin{document}

\title{ Functions of Bounded Variation on the Classical Wiener Space and an Extended Ocone-Karatzas Formula }

\author[unipi]{M.~Pratelli}

\ead{pratelli@dm.unipi.it}

\author[sns]{D.~Trevisan\corref{cor1}}

\ead{dario.trevisan@sns.it}

\address[unipi]{Dipartimento di Matematica, Universit\`a di Pisa, Largo Bruno Pontecorvo, 5, 56127, Pisa, Italy}
\address[sns]{Scuola Normale Superiore, Piazza dei Cavalieri, 7, 56126, Pisa, Italy.\\Mobile: +39 331 2899761}

\cortext[cor1]{Corresponding author}

\begin{abstract}
We prove an extension of the Ocone-Karatzas integral representation, valid for all $BV$ functions on the classical Wiener space. We establish also an elementary chain rule formula and combine the two results to compute explicit integral representations for some classes of $BV$ composite random variables.
\end{abstract}


\maketitle
\section{Introduction}

Functions of bounded variation ($BV$) in a Gaussian Banach space setting were first investigated by M.~Fukushima and M.~Hino in \cite{fuku00} and \cite{hino01}, using techniques from the theory of Dirichlet forms. More recently, L.~Ambrosio and his co-workers gave an alternative approach, in \cite{ambrosio1}, by adapting techniques from geometric measure theory.

The most important example of an infinite-dimensional Gaussian space is given by the classical Wiener space $\pa{\Omega,\mathcal{A},\Prob}$, i.e.\ the space of trajectories of the Wiener process. This was the setting where the Malliavin calculus was originally developed and still most of its applications are formulated. Since $BV$ functions generalize Malliavin differentiable functions, we specialize here the general results valid for all Gaussian spaces, work on explicit examples and consider new problems which appear naturally, in connection with stochastic analysis. Thus, all the results given here are formulated in this setting, though some of them are certainly valid in any abstract Wiener space.

One of the aims of this paper is to study how the theory of $BV$ functions can be a useful additional tool, even when dealing with classical problems. We believe that the extension of the Ocone-Karatzas formula, Theorem \ref{theorem-clark-ocone-bv}, can be regarded as the best example, in this direction, among those presented in this paper.

Roughly speaking, a real function $f$ defined on the classical Wiener space is $BV$ if it admits an $L^2\pa{0,T}$-valued measure $Df$ which plays the role of a Malliavin derivative, so that an integration-by-parts identity holds true. Generalizing the situation of differentiable functions, where the Malliavin derivative is a stochastic process, $Df$ can be seen as a measure on the product space $\pa{\Omega\times [0,T], \mathcal{A}\otimes \mathcal{B}\pa{0,T}}$ and processes can be integrated with respect to it.

Notably, $Df$ can be not absolutely continuous with respect to $\Prob\otimes\lambda$ (where $\lambda$ is the Lebesgue measure). However, if we introduce the strictly predictable $\sigma$-algebra $\mathcal{P}$, which is slightly smaller than the usual $\sigma$-algebra of predictable sets in $\Omega\times [0,T]$, then $Df$ restricted to $\mathcal{P}$ becomes absolutely continuous with respect to $\Prob\otimes\lambda$. Moreover, if $H = \pa{H_s}_{0\le s \le T}$ is a version of the density, then
\[ f = \E\qua{f} + \int_0^T H_s dW_s \punto \]
This is, informally, the content of Theorem \ref{theorem-clark-ocone-bv}, that is our extension of the classical Ocone-Karatzas formula, originally proved in \cite{karaoco91}, which identifies the integrand in the It\^o representation of a random variable in terms of its Malliavin derivative. In the differentiable case, one usually writes $H_s = \E\pa{ \partial_s f | \mathcal{F}_s }$, for $\lambda$ a.e.\ $s\in [0,T]$. Proposition \ref{proposition-predictable-projection} shows that an similar result holds true in the $BV$ case, which can be useful when dealing with concrete cases, although from a higher point of view the process $H$ should be considered as the (dual) predictable projection of the measure $Df$.

$BV$ functions can be helpful when dealing with a composite function $f = \phi\circ g$, where $g$ is a real random variable, differentiable in the Malliavin sense, and $\phi$ is a Euclidean $BV$ function. In such a case, as in the classical theory of Euclidean $BV$ functions, many problems for general functions can be reduced to the case of indicator functions of level sets $\gra{x > g}$, which can be shown to be $BV$, under certain assumptions. Moreover, an explicit chain rule formula can be obtained (Theorems \ref{theorem-chain-rule-1-level-sets} and \ref{theorem-chain-rule-2-phi}). We remark here that this chain rule is still very distant from the deep results which can be obtained in the Euclidean setting, but it can be useful when dealing with applications. 

Indeed, as an application, we combine the extended Ocone-Karatzas formula and the chain rule, to obtain explicit representations for some functionals of the Wiener process (Propositions \ref{prop-app-1} and \ref{prop-app-2}). While the former result is well-known, nevertheless we believe that the theory of $BV$ functions provides, we believe, a clear proof without advanced technical results, such as the theory of distributions on Wiener spaces.

This paper is organized as follows: in Section \ref{section-2}, we fix some notation and provide the definition of $BV$ functions together with the main approximation result, Theorem \ref{theorem-ambrosio-fukushima}, without proof. Other technical results, e.g.\ on the Orlicz space $L\log^{1/2}L$, are collected, for the convenience of the reader. In Section \ref{section-3}, we investigate a chain rule for a special class of $BV$ functions. Here, we use an approximation result for Euclidean $BV$ functions, though an elementary proof is given for the special case of indicator functions of level sets. In Section \ref{section-4}, the extended Ocone-Karatzas formula is established, after some remarks on the notion of predictability. In Section \ref{section-5}, applications and examples are discussed.

\section{Notation and preliminary results}\label{section-2}

\subsection{Malliavin calculus}

Let us fix $T>0$ and consider the classical Wiener space $\pa{\Omega, \mathcal{A}, \Prob}$, where $\Omega = C_0\pa{[0,T]}$ is the space of real continuous functions $\omega$, with $\omega\pa{0}=0$, $\mathcal{A}$ is the $\sigma$-algebra of Borel sets of $\Omega$, and $\Prob$ is the Wiener measure on $\mathcal{A}$. By definition, with respect to the probability measure $\Prob$, the canonical process $\pa{W_t}_{0\le t \le T}$, given by $W_t\pa{\omega} = \omega\pa{t}$, is a Wiener process starting from the origin.

For a complete exposition of the Malliavin calculus on the classical Wiener space, we refer to \cite{nualart1995malliavin}. Here, we recall some basic facts, together with some minor changes in the notation, which turn out to be more convenient, when dealing with $BV$ functions.

We write $L^2\pa{0,T}$ for $L^2\pa{[0,T],\lambda}$, where $\lambda$ is the Borel-Lebesgue measure, restricted to the interval $\qua{0,T}$. We write $\ang{k_1,k_2}_2$ for the scalar product between $k_1,k_2 \in L^2\pa{0,T}$.

The classical Cameron-Martin space $H^1_0 \subset \Omega$ is the space of real continuous functions $h$ on $\qua{0,T}$, such that, for some $k \in L^2\pa{0,T}$,
\[ h\pa{t} = \int_0^t k\pa{s} ds\virgola \]
for every $0\le t \le T$. Therefore $k = h'$,  $\lambda$-almost everywhere. Endowed with the scalar product $\ang{h_1,h_2}_{H_0^1} = \ang{h'_1,h'_2}_2$, $H^1_0$ is a Hilbert space isomorphic to $L^2\pa{0,T}$.

The Wiener integral construction allows us to identify $L^2\pa{0,T}$ with a subspace of $L^2\pa{\Omega, \mathcal{A}, \Prob}$, and therefore $H^1_0$ with the same subspace, with the correspondence $h \mapsto W\pa{h'} = \int_0^T h'\pa{s} dW_s$.

For $n\ge 1$, given an $n+1$-uple  of times $J = \pa{t_0, \ldots, t_n}$, with $0 \le t_0 < \ldots < t_n \le T$, we define
\begin{equation}\label{equation-delta-W} \Delta_J W = \pa{W_{t_1} - W_{t_0}, \ldots, W_{t_n} - W_{t_{n-1} } } \punto\end{equation}
We will often omit the subscript $J$ and simply write $\Delta W$. 

A smooth function $f$ is a real function of the form $f = \phi \pa{ \Delta_J W}$ for some $J$ as above and some $\phi \in C^1_b\pa{\erre^n}$. 
We remark that $\Prob$ plays no role in this definition, thus allowing us to consider different measures on $\pa{\Omega,\mathcal{A}}$.

It can be shown that any $f = \phi\pa{\Delta W}$, with a bounded continuous $\phi$, is the pointwise limit of a uniformly bounded sequence of smooth functions: indeed, it is sufficient to approximate $\phi$ with a sequence of smooth functions on $\erre^n$. Then, by the monotone class theorem, or another equivalent approximation argument, we obtain that, given a finite positive measure $\mu$, smooth functions are dense in $L^p\pa{\Omega, \mathcal{A}, \mu}$, for every $1\le p < \infty$. In particular, this holds true for $\mu = \Prob$. We will often write $L^p\pa{\Prob} = L^p\pa{\Omega, \mathcal{A}, \Prob}$.

By definition, the Malliavin derivative of a smooth function $f = \phi\pa{\Delta W}$ is the application $\nabla f: \Omega \to L^2\pa{0,T}$,
\[ \nabla f \pa{\omega}= \sum_{i=1}^n \partial_i \phi\pa{ \Delta W \pa{\omega}} I_{]t_{i-1},t_i]}  \in L^2\pa{0,T} \punto \]
The Malliavin derivative is well-defined, since
\[ \ang{\nabla f, h'}_2 = \sum_{i=1}^n \partial_i \phi\pa{\Delta W } \qua{h\pa{t_i}-h\pa{t_{i-1}} } = \partial_{h} f \punto\]
Given $h\in H^1_0$ and smooth functions $f,g$, if we write $\partial^*_h g = \partial_h g - g W\pa{h'}$, the integration-by-parts formula holds,
\[\E\qua{ \pa{\partial_h f} g } = - \E\qua{ f \partial^*_h g }\punto  \]
It follows that, for every $p\ge 1$, the linear operator $\nabla$ is well-defined on a dense subset of $L^p\pa{\Prob}$, with values in $L^p\pa{\Prob;L^2\pa{0,T}}$, and closable. We denote the domain of its closure by $\mathbb{D}^{1,p}$.

By an explicit approximation, $\nabla W\pa{k} = k$, for every $k \in L^2\pa{0,T}$. Therefore, our definition of smooth functions provides a construction of the Malliavin derivative, equivalent to that in \cite{nualart1995malliavin}.

The fact that the Malliavin derivative of a random variable can be identified with a process is a consequence of the following elementary result, which can be easily proved by a density argument. We prefer to state it as a lemma, since it will be used again when dealing with $BV$ functions.

\begin{lemma}\label{lemma-measure-l2-measure-product}
Given a positive measure $\nu$ on $\pa{\Omega,\mathcal{A}}$, there is a linear continuous immersion
\[ L^1\pa{\Omega,\nu; L^2\pa{0,T}} \to L^1\pa{\Omega\times [0,T], \mathcal{A}\otimes \cB\pa{[0,T]}, \nu\otimes \lambda} \]
that maps every $F$ to a process $F^\nu$ such that, $\nu$ almost everywhere, the function $t  \mapsto F^\nu_t\pa{\omega}$ coincides with $F\pa{\omega}$, $\lambda$ almost everywhere. Moreover,
\[ \int_{\Omega\times [0,T]} \abs{ F^\nu } d\pa{ \mu\otimes \lambda} \le T^{1/2}\int _\Omega \abs{F }_2 d\nu \punto \]
\end{lemma}

Therefore, given $f \in \cD^{1,1}$, we write $\pa{\partial_t f}_{0\le t \le T}$ for the process $\nabla f ^\Prob$ given by the lemma above. For every $k \in L^2\pa{0,T}$,
\[ \ang{\nabla f, k}_2 = \int_0^T \partial_s f \, k\pa{s} ds \virgola \quad \textrm{  $\Prob$ a.s. }  \]

\subsection{ The space $L \log^{1/2} L$ }

We write $A_{1/2}$ for the real continuous convex function
\[ x \mapsto A_{1/2}\pa{x} = \int_0^{\abs{x}} \log^{1/2}\pa{1+s} ds \punto\]

By definition, $L \log^{1/2}L = L\log^{1/2}L\pa{\Prob}$ is the vector space of real random variables $X$, such that, for some $k>0$, $A_{1/2}\pa{X/k}\in L^1\pa{\Prob}$. Endowed with the Luxembourg norm
\[ \norm{X}_{L \log^{1/2}L }  = \inf\gra{k>0 \, : \, \E\qua{A_{1/2}\pa{X/k}} \le 1 } \virgola\]
it is a particular case of an Orlicz space, and therefore a Banach space.

Orlicz spaces generalize $L^p$ spaces and, if the function which defines the norm does not grow too fast, many properties can be stated and proved in exactly the same way. $L\log^{1/2} L$ is such an example of slow growth, since for every real $0 < k_1 < k_2 $ and every real $x$,
\[ \log^{1/2}\pa{1+x/k_1}/k_1 < \pa{k_2/k_1}^{3/2} \log^{1/2}\pa{1 +  x/k_2}/k_2 \punto\]
It follows that
\[ \E \qua{A_{1/2} \pa{X/k_1} } \le \pa{k_2/k_1}^{3/2} \E \qua{A_{1/2} \pa{X/k_2} } \virgola\]
and so, if $X \in L \log^{1/2} L$ then, for every $k>0$, $A_{1/2}\pa{X/k}$ is integrable.

Given $X \in L \log^{1/2} L$ and a real random variable $Y$ with centred normal law then, for some $C\pa{Y}$,
\begin{equation}\label{equation-holder-orlicz} \E\qua{\abs{ X Y} } \le C\pa{Y} \norm{X} _{L \log^{1/2} L} \punto \end{equation}
This follows, for example, from Young's inequality for $A_{1/2}$,
\[ \abs{xy} \le A_{1/2} \pa{x} + \int_0^{\abs{y}}\pa{e^{t^2} -1} dt \le A_{1/2} \pa{x} + \abs{y} e^{y^2} \virgola\]
and taking expectation, with $x = X/\norm{X} _{L \log^{1/2}L }$ and $y = Y/2\sqrt{\E\qua{Y^2}}$.

In the context of Malliavin differentiable functions, an important consequence of the result above is that, given a smooth function $g$ and $h \in H^1_0$, $\E\qua{f \partial^*_h g}$ is well-defined whenever $f \in L \log^{1/2} L$.

\begin{remark}\label{remark-embedding-d1-llog}
Another fundamental fact, of which we will make implicit use, since it is partially contained in Theorem \ref{theorem-ambrosio-fukushima}, is the existence of a continuous embedding of $\mathbb{D}^{1,1}$ into $L \log^{1/2} L$ (see \cite{hino01}, Proposition 3.2). We remark that, in turn, this is a consequence of the Gaussian isoperimetric inequality (see \cite{springerlink:10.1007/BFb0095676}).
\end{remark}

We turn now to a technical result concerning the convergence of closed martingales in the space $L \log^{1/2} L$, which will be used in the proof of Theorem \ref{theorem-clark-ocone-bv}.

\begin{lemma}\label{lemma-martingale-llog}
Given a discrete filtration $\mathcal{G} = \pa{\mathcal{G}_n}_{n\ge1}$, with $\bigvee _n G_n = \mathcal{A}$, and a $\mathcal{G}$-martingale $\pa{M_n}_{n\ge1}$, closed by a random variable $M \in L\log^{1/2} L$ (so that for every $n\ge 1$, $M_n = \E\qua{ M | \mathcal{G}_n}$), then the sequence $\pa{M_n}_{n\ge 1}$ converges to $M$ in $L\log^{1/2} L$. 
\end{lemma}

\begin{proof}
It is well known that $\pa{M_n}_{n\ge1}$ converges $\Prob$ a.s.\ to $M$ and therefore in probability. We fix $\epsilon >0$  and we show that
\[ \limsup_n \E\qua{ A_{1/2}\pa{ \frac{M_n - M}{\epsilon} } } < 1 \punto\]
Since $A_{1/2}\pa{0} = 0$, by continuity, there exists some $\delta > 0$ such that $A_{1/2}\pa{x}  \le 1/2$ if $\abs{x} \le \delta$. We write $B\pa{n}  = \gra{ \abs{M_n - M} > \delta \epsilon }$, so that
\[ \E\qua{ A_{1/2}\pa{ \frac{M_n - M}{\epsilon} } }  \le \frac1 2 + \E\qua{I_{B\pa{n} } A_{1/2}\pa{ \frac{M_n - M}{\epsilon} } } \punto\]
Moreover, $\lim_n \Prob\pa{B\pa{n}}=0$. By the elementary properties of $A_{1/2}$,
\[ A_{1/2}\pa{ \frac{M_n - M}{\epsilon} } \le \frac 1 2 A_{1/2}\pa{ \frac{2M}{\epsilon} }  + \frac 1 2 A_{1/2}\pa{ \frac{2M_n}{\epsilon} } \punto\]
$M$ belongs to $L\log^{1/2} L$, so the first summand above is integrable and
\[ \limsup_n \E\qua{I_{B\pa{n} } A_{1/2}\pa{ \frac{2M}{\epsilon} } } = 0 \punto \]
Since $A_{1/2}$ is convex, the second summand above is $\Prob$ a.s.\ not greater than $1/2 \E\qua{ A_{1/2}\pa{ 2M/\epsilon} | \mathcal{G}_n }$, by Jensen's inequality for conditional expectations. Again, $A_{1/2}\pa{ 2M/\epsilon}$ is integrable so $\pa{\E\qua{ A_{1/2}\pa{ 2M/\epsilon} | \mathcal{G}_n }}_{n\ge1}$ is a $\mathcal{G}$-martingale closed in $L^1$ and therefore uniformly integrable. We conclude that
\[ \limsup_n\E\qua{I_{B\pa{n} } A_{1/2}\pa{ \frac{2M_n}{\epsilon} } } \le \limsup_n\E\qua{I_{B\pa{n} } \E\qua{ A_{1/2}\pa{\frac{2M}{\epsilon}} | \mathcal{G}_n } } = 0 \punto\]\end{proof}

\subsection{ $L^2\pa{0,T}$-valued measures }

We write $\mathcal{M} = \mathcal{M}\pa{\Omega; L^2\pa{0,T} }$ for the space of  $L^2\pa{0,T}$-valued $\sigma$-additive measures on $\pa{\Omega,\mathcal{A}}$, with finite total variation $\abs{\mu}$. We recall that $\abs{\mu}$ is a measure on $\pa{\Omega,\mathcal{A}}$, given by
\[ \abs{\mu} \pa{A} := \sup\gra{ \sum_{n\ge1} \abs{\mu\pa{A_n}}_2: A = \bigcup_{n\ge1} A_n } < \infty \virgola \]
where the supremum runs along every countable measurable partition of $A$.

By the polar decomposition theorem, given $\mu \in \mathcal{M}$, there exists an $L^2\pa{0,T}$-valued measurable application $\sigma$, with $\abs{\sigma}\pa{\omega} \le 1$,  for all $\omega$, such that, for every measurable set $A$, and every $k \in L^2\pa{0,T}$,
\[ \ang{\mu\pa{A},k}_2 = \int I_A \ang{\sigma,k}_2 d\abs{\mu}\punto \]
The member on the right clearly defines a real measure, which we denote by $\ang{\mu, k}_2$. The decomposition above allows us to integrate $L^2\pa{0,T}$-valued applications. We define
\[ \int \ang{F, d\mu}_2 = \int \ang{\sigma, F}_2 d\abs{\mu}\virgola \]
if $\ang{\sigma, F}_2 \in L^1\pa{\abs{\mu}}$. Since such applications can be seen as processes, we identify $\mu$ with a real measure on the product space. More precisely, we apply Lemma \ref{lemma-measure-l2-measure-product} with $\nu = \abs{\mu}$ and consider the measure $\tilde{\mu} = \sigma^{\abs{\mu}} . \pa{\abs{\mu} \otimes \lambda}$. Given a process $F = f I_{]s,t]}$, with smooth $f$ and $0\le s \le t \le T$, then
\[ \int \ang{F, d\mu}_2 = \int f d\ang{\mu, I_{]s,t]} }_2 = \int_\Omega f\pa{\omega}\qua{ \int_s^t \sigma^{\abs{\mu}}\pa{\omega,r}dr }\abs{\mu}\pa{d\omega} \punto \]

\subsection{$BV$ functions}

The following condition shows that the integration-by-parts formula plays a central role in the theory of Malliavin calculus: $f\in \cD^{1,p}$, with $1 \le p < \infty$, if and only if $f \in L^p\pa{\Prob}$  and there exists some $F \in L^p\pa{\Omega,\Prob; \ldue}$ such that, for every smooth $g$ and every $h\in H^1_0$, 
\begin{equation} \label{equation-integration-by-parts} \E\qua{ \ang{ F, h'}_2 g } = - \E\qua{ f \partial^*_h g } \punto \end{equation}
In such a case, $F = \nabla f$.

We can read the left member in \eqref{equation-integration-by-parts} as the integral of $g$, with respect to the real measure $\ang{F,h'}.\Prob$. Informally, a function $f$ is said to be of bounded variation, if we require only that there exists a measure, such that the same condition is satisfied.

\begin{definition}
A real function $f$ is said to be of \emph{bounded variation} ($BV$), with respect to $\Prob$, if $f \in L\log^{1/2}L \pa{\Prob}$, and there exists a measure $Df \in \mathcal{M}$, such that, for every $h \in H^1_0$  and every smooth $g$,
\[ \int g \, d\ang{Df,h'}_2 = - \E\qua{f \, \partial_h ^*g }\punto\]
The quantity  $\abs{Df}\pa{\Omega}$ is called the \emph{total variation} of $f$. The real measure $\ang{Df,h'}_2$ is denoted with  $D_h f$.
\end{definition}

\begin{remark}
In the article \cite{ambrosio1}, it is required for a smooth function to be of the form $\phi \pa{W\pa{k_1}, \ldots, W\pa{k_n}}$, where $\phi$ is smooth and each $k_i \in \ldue$ is a function of bounded variation on $[0,T]$, so that $W\pa{k_i}$ is linear and continuous on $\Omega$. Therefore, although the definitions are formally identical, the class of $BV$ functions introduced above could be larger than that considered there. However, by an approximation argument, it can be shown that they coincide.\end{remark}

Any function $f \in \mathbb{D}^{1,1}$ is $BV$, with $Df = \nabla f. \Prob$, because the integrability condition follows from the continuous embedding of $\mathbb{D}^{1,1}$ in $L\log^{1/2}L\pa{\Prob}$ (see Remark \ref{remark-embedding-d1-llog}). Actually, a $BV$ function $f$ admits a Malliavin derivative if and only if $\abs{Df}$ is absolutely continuous with respect to $\Prob$.

As a consequence of the general results about $L^2\pa{0,T}$-valued measures, in all what follows we will identify the measure $Df$ with a measure on the product space $\Omega\times [0,T]$. We can even define a $BV$-analogue of the process $\pa{\partial_t f}_{0\le t \le T}$: given a version $\sigma$ of the density of $Df$ with respect to $\abs{Df}\otimes\lambda$,  we define $D_tf = \sigma\pa{t} . \pa{\abs{Df}\otimes\lambda}$. The family of measures $\pa{D_tf}_{0\le t \le T}$ is defined up to $\lambda$-neglegible sets.

We consider now the approximability of $BV$ functions with regular functions. To make a comparison with the differentiable case, we recall that, for $p>1$, given a sequence $\pa{f_n}_{n\ge1}$ in $\cD^{1,p}$, convergent to some $f$ in $L^p\pa{\Prob}$ and such that $\pa{\nabla f_n}_{n\ge1}$ is bounded in $L^p\pa{\Omega, \Prob;\ldue}$, then $f \in \cD^{1,p}$ (see \cite{nualart1995malliavin}, Lemma 1.2.3, p.\ 28, for the case $p=2$). Such a conclusion does not hold, for $p=1$: however, the next theorem shows that $f$ must be a $BV$ function, and all $BV$ functions can be obtained with a similar approximation. Proofs can be found in \cite{hino01}, Theorem 3.7, or in \cite{ambrosio1}, Theorem 4.1.

\begin{theorem}\label{theorem-ambrosio-fukushima}
Given $f \in L^1\pa{\Prob}$, the following conditions are equivalent:
\begin{enumerate}
\item $f$ is of bounded variation;
\item there exists a sequence of functions $\pa{f_n}_{n\ge 1}$, bounded in $\mathbb{D}^{1,1}$ and convergent to $f$ in $L^1\pa{\Prob}$;
\end{enumerate}
In such a case,
\[ \abs{Df}\pa{\Omega} \le \liminf_{n} \E\qua{ \abs{\nabla f_n}_2} \virgola\]
and there exists a sequence $\pa{f_n}_{n\ge1}$ such that equality is attained.
\end{theorem}

We remark that in the theorem above, the second condition does not mention the space $L\log^{1/2}L$: this follows by the continuous embedding discussed in Remark \ref{remark-embedding-d1-llog}. In this sense, the extra integrability condition required in the definition of $BV$ functions is technical but natural.

\section{A chain rule}\label{section-3}

For a composite function $f = \phi \circ g$, under certain assumptions on $\phi$ and $g$, we can conclude that $f$ is $BV$ and write an explicit formula for the integral of smooth functions with respect to any measure $D_hf$.

A chain rule for the Malliavin derivative (see \cite{nualart1995malliavin}, Proposition 1.2.3, p.\ 28) reads as follows: given $g_1,\ldots,g_n \in \cD^{1,1}$ and $\phi \in C^1_b\pa{\erre^n}$, then $f=\phi\pa{g_1,\ldots,g_n} \in \cD^{1,1}$ and
\[ \nabla f =  \sum_{i=1}^n \partial_i\phi\pa{g_1,\ldots,g_n} \nabla g_i \punto\]

When $\phi$ is a Euclidean $BV$ function, we would like to conclude that $f$ is $BV$ (for the general theory of Euclidean $BV$ functions, see \cite{ambfus00}). Before stating some precise results, we give here a formal derivation of the chain rule in the simple case of $g_i = W\pa{k_i}$, with $\gra{k_1, \ldots, k_n}$ orthonormal in $L^2\pa{0,T}$. Under this assumption, the joint law of $g = \pa{W\pa{k_1},\ldots, W\pa{k_n} }$ is the standard Gaussian law on $\erre^n$ and we write $\rho$ for its continuous density. Given a bounded random variable $u$, there exits some Borel function $v$ such that $\E\qua{u | g } = v\pa{g }$. For brevity, we write $v\pa{x}= \E\qua{u | g = x}$. For $h \in H^1_0$,  when $\phi \in C^1_b\pa{\erre^n}$, we integrate the chain rule above:
\[ \E\qua{ u \ang{h', \nabla f} } = \sum_{i=1}^n \E\qua{ u \partial_i\phi\pa{g} \ang{k_1,h'}_2 } \punto\]
The left member above can be replaced with $\int u D_h f$, which is defined when $f$ is $BV$. For the right member, we take the conditional expectation with respect to $g$ and we find that
\[ \int u D_h f = \sum_{i=1}^n\int_{\erre^n} \E\qua{u | g = x } \ang{k_1,h'}_2 \rho\pa{x} \partial_i\phi\pa{x} \lambda^n\pa{dx} \punto\]
When $\phi$ is a Euclidean $BV$ function, we would like to replace $\partial_i\phi\pa{x} \lambda^n\pa{dx}$ with $D_i\phi\pa{dx}$. However, we notice that the integrand in the right member above is defined $\lambda^n$ a.e.\ and, since $\abs{D_i\phi}$ can be singular with respect to $\lambda^n$, it is clear that some assumptions on $u$ are necessary to give a precise meaning to the expression. In this setting, the following theorem can be established, but we omit the proof.

\begin{theorem}
With the notation as above, if $u$ is cylindrical and continuous, then there exists a continuous representative of  $\E\qua{u | g= x }$ and, if $\rho. \abs{D\phi}$ is a finite measure, then $f$ is $BV$ and
\[  \int u D_h f = \sum_{i=1}^n\int_{\erre^n} \E\qua{u | g = x } \ang{k_1,h'}_2 \rho\pa{x} D_i\phi\pa{dx} \virgola\]
where the integrand is intended as its continuous representative.
\end{theorem}

By a density argument, the formula above identifies the measure $Df$ as an expression of $D\phi$ and $\nabla W\pa{k_i} = k_i$.

We turn now our attention to some chain rule formulas valid with less regularity assumptions. In what follows, we consider the case $n=1$ only, but we do not limit ourselves to a Wiener integral: more precisely, we assume that $g$ belongs to $\cD^{1,1}$ and moreover that its law is absolutely continuous with respect to $\lambda$, with a locally bounded density $\rho$. 

Given an integrable random variable $X$, we write
\[ x \mapsto \E\qua{ X | g = x} \virgola\]
for the equivalence class in $L^1\pa{\rho.\lambda}$ of any Borel function $h: \erre \to \erre$ such that $h\circ g = P\qua{X \,|\, g}$.

We introduce the following notation:
\[\xi\pa{x} = \rho\pa{x} \E\qua{  \abs{\nabla g}_2 | g = x} \virgola \]
which defines an element of $L^{1}\pa{\lambda}$. We assume that $\xi$ is locally bounded and write $\xi_*\pa{x} = \liminf_{n\to \infty} n\int_{x}^{x+1/n} \xi\pa{t}dt$.

\begin{theorem}\label{theorem-chain-rule-1-level-sets}
For every real $x$, the level set indicator function $f^x = I_{]x,\infty[}\pa{g} = I_{ \gra{g > x} }$ is $BV$ and $\abs{Df^x}_2\pa{\Omega} \le \xi_*\pa{x}$.

Let $u$ be a smooth cylindrical function and $h \in H^1_0$. Then, for $\lambda$ a.e.\ $x$,
\[  \int u D_h f^x = \rho\pa{x} \E\qua{ u \ang{\nabla g, h'}_2 | g = x }\punto \]
\end{theorem}

\begin{remark}
The left member in the identity above is a continuous function of $x$, since it coincides with $-\E\qua{f^x \partial^*_h u }$ and the law of $g$ has no atoms. Therefore, the right member above admits a continuous representative.
\end{remark}

\begin{proof}
For $x \in \reals$, for $n\ge 1$, we define
\[ \phi^x_n \pa{t} = \pa{n(t-x) \land 1} \lor 0 \virgola\]
and write $f^x_n = \phi^x_n \circ g$. The sequence $\pa{f^x_n}_{n\ge1}$ converges to $f^x$ in every $L^p\pa{\Prob}$, for $p\ge 1$. By the classical chain rule, $f^x_n$ admits a Malliavin derivative given by
\[ \nabla f^x_n = n I_{]x, x+1/n[}\pa{g} \nabla g \punto\]
Then, conditioning with respect to $g$,
\[ \E\qua{\abs{ \nabla f^x_n  }_2 } = n \int_x^{x+1/n} \xi\pa{t} dt \virgola\]
which is bounded for $n \to \infty $. By Theorem \ref{theorem-ambrosio-fukushima}, the first statement is proved.

Given $u$ and $h$ as in the second statement, we write the integration-by-parts identity for every $f^x_n$:
\[ - \E\qua{ \partial^*_h u f^x_n} = \E\qua{ u \ang{\nabla f^x_n, h'}_2 }  \punto\]
We take the conditional expectation with respect to $g$ and integrate with respect to its law,
\[ - \E\qua{ \partial^*_h u f^x_n} = n \int_x^{x+1/n}\rho \pa{t} \E\qua{u \ang{\nabla g, h'}_2 | g = t } dt \punto\]
As $n \to  \infty$, the right member above converges to $\rho \pa{x} \E\qua{u \ang{\nabla g, h'}_2 | g = x }$, for $\lambda$ a.e.\ $x$, while the left member converges to $-\E\qua{ \partial^*_h u f^x}$, for every $x$.
\end{proof}

Once the case of level sets is settled, we consider a $\phi \in BV_{loc}\pa{\lambda}$ such that and $\abs{D\phi}$ has compact support, contained in an interval $[-M, M]$, for some $M>0$. In particular, such a $\phi$ is $\lambda$ a.e.\ constant for $\abs{x}$ large (possibly with two different values according to the sign of $x$). Therefore, $\phi$ differs from a Euclidean $BV$ function with compact support contained in the same interval $[-M, M]$, by adding a suitable constant function and a multiple of an indicator function of an unbounded interval. We write $C_M$ for a real number such that $\xi\pa{x} \le C_M$ for $\lambda$ a.e.\ $x \in [-M, M]$.

\begin{theorem}\label{theorem-chain-rule-2-phi}
With the notation above, $f = \phi \circ g$ is $BV$, with $\abs{Df}\pa{\Omega} \le C_M \abs{D\phi}\pa{\erre}$.

Let $u$ be a smooth cylindrical function and $h \in H^1_0$. Then,
\begin{equation} \label{equation-bv-chain-rule} \int u D_h f = \int_\erre \rho\pa{x} \E\qua{ u \ang{\nabla g, h'}_2 | g = x } D\phi\pa{dx} \end{equation}
where the integrand on the right is intended as its continuous representative.
\end{theorem}

\begin{proof}
We can suppose that $\phi \in BV\pa{\lambda}$ has compact support. By the Euclidean analogue of Theorem \ref{theorem-ambrosio-fukushima} (see \cite{ambfus00}, Theorem 3.9) there exists a sequence of smooth functions with compact support $\pa{\phi_n}_{n\ge1}$, convergent to $\phi$ in $L^1\pa{\lambda}$, such that the sequence of derivatives $\pa{\phi'_n } _{n\ge1}$ is bounded in $L^1\pa{\lambda}$. We define $f_n = \phi_n \circ g$, so that, for some $C>0$,
\[ \E\qua{\abs{f_n - f} } = \int \abs{\phi_n -\phi} \rho d\lambda \le C \norm{\phi_n - \phi }_{L^1} \punto \]
By the chain rule for differentiable functions, for every $n\ge 1$,
\[ \nabla f_n = \phi'_n\pa{g} \cdot \nabla g \punto\]
Taking expectations and conditioning with respect to $g$,
\[ \E\qua{\abs{\nabla f_n}_2} = \E\qua{ \abs{\phi'_n\pa{g}} \E \qua{\abs{\nabla g}_2 \,|\, g } } \punto\]
Integrating with respect to the law of $g$, the right member above is equal to
\[ \int_\erre  \abs{\phi'_n\pa{x}}\xi\pa{x} dx   \le C_M  \int_\erre  \abs{\phi'_n\pa{x}}dx\virgola \]
and we conclude that $f$ is $BV$.

To obtain the chain rule \ref{equation-bv-chain-rule}, we can suppose that the sequence of measure derivatives $\pa{D\phi_n}_{n\ge1}$ (where $D\phi_n = \phi'_n. \lambda$) converges in the duality with bounded continuous functions. It is not hard to see that equation \ref{equation-bv-chain-rule} holds true with $\phi_n$ instead of $\phi$ and therefore holds true in the limit, if the integrand is intended as its continuous representative.
\end{proof}

\begin{remark}
If $\xi$ admits a continuous representative, we obtain the stronger estimate $\abs{Df}\pa{\Omega} \le \int_{\erre} \xi\pa{x} \abs{D\phi}\pa{dx}$.
\end{remark}

In the last section we will be given a Borel function $U\pa{x,r}$, such that, for some $s\in [0,T]$ and $\lambda$ a.e.\ $r \ge s$, $U\pa{x,r}$ is a version of $\rho\pa{x} \E\qua{ u \partial_rg | g = x }$, continuous in $x$. In order to obtain a chain rule in terms of $U$, we make some additional boundedness assumptions. 

\begin{theorem}\label{theorem-chain-rule-4-disintegrated-useful}
Assume that, for every real $x$, $r \mapsto U\pa{x,r}$ belongs to $L^2\pa{0,T}$, and the function $x \mapsto \abs{U\pa{x,\cdot}}_2$ is locally bounded. Then, for $\lambda$ a.e.\ $r \ge s$,
\begin{equation}\label{equation:bv:disintegrated} \int u D_r f  = \int_\erre U \pa{x,r} D\phi\pa{x}\punto \end{equation}
\end{theorem}

\begin{proof}

The following elementary fact will be useful. Let $\pa{k_n}_{n\ge1}$ be a bounded sequence in $L^2\pa{0,T}$ convergent $\lambda$ a.e.\ to some function $k$. Then $k \in L^2\pa{0,T}$ and the sequence converges weakly to $k$.

Because of the assumptions and the fact stated above, for every $h\in H^1_0$ with $h\pa{r} = 0$ if $r<s$, $\int_0^T U\pa{x,r} h'\pa{r} dr$, is well defined and continuous as a function of $x$. Moreover, if $\psi \in C_c\pa{\erre}$, then 
\[ \int_\erre \psi \pa{x} \int_0^T U\pa{x,r} h'\pa{r} dr dx = \int_0^T \E\qua{ u \pa{\psi \circ g }\partial_r g } h'\pa{r} dr \virgola \]
since for every $r\ge s$, $U\pa{x,r}$ is a version of $\rho\pa{x} \E\qua{ u \partial_rg | g = x }$. We can exchange integration and expectation above, because $u\pa{\psi \circ g}$ is bounded, and we conclude that $\int_0^T U\pa{x,r} h'\pa{r} dr = \rho\pa{x} \E\qua{ u \ang{ \nabla g, h'}_2 | g = x }$ for $\lambda$ a.e.\ $x$. This suffices to settle the case of indicator functions of level sets.

For the case of a function $\phi$, which we suppose of compact support, we consider again an approximation with smooth functions $\pa{\phi_n}_{n\ge1}$ with derivatives of compact support, such that $\pa{ D\phi_n}_{n\ge1}$ weakly converges to $D\phi$. For every $n$, and every $h\in H^1_0$, with $h'\pa{r} =0$ if $r<s$,
\[ \int_0^T h'\pa{r} \int_\erre U\pa{x,r} D\phi_n\pa{x} dr = \int_\erre \rho\pa{x} \E\qua{ u  \ang{ \nabla g, h'}_2| g = x } D\phi_n\pa{x} \virgola\]
Then, for $\lambda$ a.e.\ $r\ge s$, $\int_\erre \rho\pa{x} \E\qua{ u \partial_r g | g = x } D\phi_n\pa{x}$ tends to the left member in equation \eqref{equation:bv:disintegrated}, since the integrand is continuous. Moreover, the sequence is bounded in $L^2\pa{s,T}$, since the integrand is locally bounded,
\[ \abs{\int_\erre \rho\pa{x} \E\qua{ u \partial_r g | g = x } D\phi_n\pa{x}}_2 \le C  \abs{D\phi_n}\pa{\erre} \punto \]
As $n \to \infty$, the fact stated at the beginning of the proof applies again and we conclude that
\[ \int_0^T h'\pa{r} \qua{ \int_\erre U\pa{x,r} D\phi\pa{x}}  dr = \int u D_h f  = \int_0^T h'\pa{r} \qua{\int u D_r f} dr\punto \]
\end{proof}

\section{Ocone-Karatzas formula for $BV$ functions}\label{section-4}

\subsection{Predictable processes and projections}

Given a filtration $\mathcal{G} = \pa{\mathcal{G}_t}_{0\le t \le T}$, a $\mathcal{G}$-predictable rectangle is a subset of the product space $\Omega \times [0,T]$, of the form
\[ A \times ]s,t] \virgola\]
where $A \in \mathcal{G}_s$. The $\mathcal{G}$-predictable $\sigma$-algebra of sets $\mathcal{P}_{\mathcal{G}}$ is, by definition, generated by the family of all $\mathcal{G}$-predictable rectangles. A stochastic process is said $\mathcal{G}$-predictable if it is measurable with respect to $\mathcal{P}_{\mathcal{G}}$.

Usually, one takes $\mathcal{G} = \mathcal{F}^\Prob$, the natural filtration of the Wiener process, completed with all the $\Prob$-neglegibile Borel sets, thus satisfying the so-called usual conditions, and simply speaks of predictable rectangles, predictable $\sigma$-algebra ($\mathcal{P}^\Prob$) and predictable processes. Since we are dealing with measures, in general, not absolutely continuous with respect to $\Prob$, we consider the case when $\mathcal{G} = \mathcal{F}$ is just the natural filtration of the Wiener process, and speak of \emph{strictly} predictable rectangles, \emph{strictly} predictable $\sigma$-algebra ($\mathcal{P}$) and \emph{strictly} predictable processes, for the correspondent $\mathcal{F}$-predictable objects. 

By a monotone class argument, every predictable process coincides $\Prob\otimes \lambda$ almost everywhere with a strictly predictable process.

A smooth process is a finite linear combination of strictly predictable processes of the form
\[ f I_{]s,t]}\virgola \]
where $f$  is smooth. The strict predictability implies that $f$ can be written as $\phi\pa{\Delta_J W}$ for some $\Delta_J W$, as introduced in \eqref{equation-delta-W}, such that $s_n \le s$. By a direct approximation of cylindrical functions with smooth functions and a monotone class argument, smooth processes are dense in every $L^p\pa{\Omega\times [0,T], \mathcal{P}, \mu}$, for $1\le p < \infty$, where $\mu$ is any positive finite measure on $\mathcal{P}$. Therefore, if we are given two real measures $\mu$, $\nu$ on $\mathcal{P}$ such that for every smooth process $F$, $\int F d\mu = \int F d\nu$, then $\mu = \nu$.

The classical Ocone-Karatzas formula contains the stochastic integral of a predictable process $H = \pa{H_s}_{0\le s \le T}$, which can be correctly defined as the density of the measure $\nabla f. \pa{\Prob\otimes \lambda}$ restricted to the strictly predictable $\sigma$-algebra $\mathcal{P}$, with respect to $\Prob\otimes \lambda$. The process $H$ is also identified by the condition $H_s = \E\qua{\partial_s f| \mathcal{F}_s}$, for $\lambda$ a.e.\ $s \in [0,T]$. We can see $H_s$ as the density of $\partial_s f. \Prob$, restricted to $\mathcal{F}_s$, with respect to $\Prob$ (also restricted).

Given a $BV$ function $f$, Theorem \ref{theorem-clark-ocone-bv} will prove that its measure derivative $Df$, as a measure on the product space, when restricted to $\mathcal{P}$, becomes absolutely continuous with respect to $P\otimes \lambda$, with a strictly predictable density $H = \pa{H_s}_{0\le s \le T}$. The next lemma shows that we can still see $H_s$ as the density of the measure $D_sf$, restricted to $\mathcal{F}_s$, with respect to $\Prob$. We state it for a general real measure $\mu$ on the product space $\pa{\Omega \times [0,T], \mathcal{A}\otimes \mathcal{B}\pa{0,T}}$ of the form $\mu = \sigma . \pa{\nu \otimes \lambda}$, where $\nu$ is a positive finite measure on $\pa{\Omega, \mathcal{A}}$.

\begin{proposition} \label{proposition-predictable-projection}
With the notation above, if $\mu_{|\mathcal{P}}$ is absolutely continuous with respect to $\pa{P\otimes \lambda}_{|\mathcal{P}}$, with a strictly predictable density $H = \pa{H_s}_{0\le s \le T}$, then, for $\lambda$ a.e. $s\in [0,T]$,
\[  \pa{\sigma\pa{s} .\nu}_{|\mathcal{F}_s} = H_s. P_{|\mathcal{F}_s} \punto \]
\end{proposition}

\begin{proof}
For every $n\ge1$, there exists a countable family of bounded continous function $\pa{f^n_k}_{k\ge1}$, defined on $\erre^n$, such that every $f \in C_c\pa{\erre^n}$ is the pointwise limit of an appropriate sequence with elements from that family. Therefore, for $\lambda$ a.e.\ $s \in [0,T]$, it will suffice to prove that, for all $n\ge1$, for every $\Delta_J W$, with $J = \pa{t_0, \ldots, t_n}$ with $t_n \le s$ and every $k \ge 1$,
\[ \int f^n_k\pa{\Delta_J W}  \sigma\pa{s} d\nu = \E\qua{ f^n_k\pa{\Delta_J W} H_s } \punto\]
Moreover, by continuity, we can consider only the case with each $t_i$ rational. Therefore, the thesis is equivalent to prove that for every $n$, every $\Delta_J W$, with each $t_i$ rational, and every $k \ge 1$, for $\lambda$ a.e.\ $s\ge t_n$, the identity above holds true. The hypothesis implies that for every $t \ge t_n$,
\[ \int_{t_n}^t \int f^n_k\pa{\Delta_J W}  \sigma\pa{s} d\nu ds= \int_{t_n}^t\E\qua{ f^n_k\pa{\Delta_J W} H_s } ds \punto \]
since it is the integral of a stricly predictable process. Therefore the two integrands must coincide for $\lambda$ a.e\ $s\ge t_n$.
\end{proof}

\subsection{Statement and proof of the main result}

For the proof of the extended Ocone-Karatzas formula we will make use of the following result, which contains It\^o's representation theorem and a version of Ocone-Karatzas formula for smooth functions, which can be shown directly by an application of It\^o's formula.

\begin{lemma}
For $f \in L^1\pa{\Prob}$, there exists a unique process $F \in \mathcal{M}^2_{loc}$ such that
\begin{equation}\label{equation-ito-representation} f = \E\qua{f} + \int_0^T F_s dW_s \punto\end{equation}
If $f$ is smooth, then $F_s = \E\qua{\partial_s f | \mathcal{F}_s}$, for $\lambda$ a.e. $s \in [0,T]$.
\end{lemma}

\begin{remark}
The space $\mathcal{M}^2_{loc}$ consists of all the equivalence classes of predictable processes $F$, such that
\[ \int_0^T F_s^2 ds < \infty \virgola  \quad \textrm{ $\Prob$ almost surely} \punto\]
For such an $F$, there exists a sequence of predictable stopping times $\pa{\tau_n}_{n\ge1}$, increasing towards $T$, such that, for every $n$, the predictable process  $F_n = F I_{[0,\tau_n[}$ is square integrable (and so it belongs to $\mathcal{M}^2$).
\end{remark}

\begin{theorem}\label{theorem-clark-ocone-bv} Let $f$ be a $BV$ function. The measure $Df$, restricted to the strictly predictable $\sigma$-algebra $\mathcal{P}$, is absolutely continuous with respect to $\Prob \otimes \lambda$. If $H$ is a version of the density, then
\[f = \E\qua{f} + \int_0^T H_s dW_s \punto\]
\end{theorem}

\begin{remark}\label{remark-H-density}
By Proposition \ref{proposition-predictable-projection}, we can write, for $\lambda$ a.e.\ $s\in[0,T]$:
\[ H_s  = \frac{ d\pa{D_s f}_{|\mathcal{F}_s} }{ dP_{|\mathcal{F}_s} } \punto \]
\end{remark}

We provide here a shorter proof for case $f \in BV \cap L^2\pa{\Prob}$.

\begin{proof} There is no loss in generality if we suppose that $\E\qua{f} = 0$. Since $f$ is square integrable, the process $F$ in \eqref{equation-ito-representation} belongs to $\mathcal{M}^2$ and therefore $F.\Prob \otimes \lambda$ defines a real measure on the product space. It is sufficient to show that the integrals of smooth processes, taken with respect to $Df$ and $F.\Prob\otimes \lambda$, coincide.

Given $0\le s<t \le T$ and a smooth function $g$, measurable with respect to $\mathcal{F}_s$, the Wiener-It\^o isometry implies that
\[\E\qua{f\int_s^t g  dW_r }= \E\qua{ \int_0^T  F_r dW_r \int_s^t g  dW_r } = \E\qua{ \int_s^t  g  F_r dr }  \punto \]
Since $g$ is $\mathcal{F}_s$-measurable, $\int_s^t g dW_r  = -\partial^*_h g$, where $h \in H^1_0$ is such that $h' = I_{]s,t]}$. By the definition of $Df$, the left member above is equal to
\[ \int g I_{]s,t]} \, d Df\punto \]
By linearity, we conclude.
\end{proof}

The next lemma shows that we gain integrability for $F$ in \eqref{equation-ito-representation}, when $f$ is $BV$.

\begin{lemma}
If $f$ is $BV$ and $F$ is the process in \eqref{equation-ito-representation}, then
\[ \int_0^T \E\qua{\abs{F_s} } ds \le  T^{1/2}\abs{Df} \pa{\Omega} \punto\]
\end{lemma}

\begin{remark}
If $f \in L\log L$, then, by the continuous embedding of this Orlicz space into $H^1$ (see e.g.\ \cite{revuz1999continuous}, Exercise 1.16, p.\ 58) and by the equivalence of the $H^1$ norm with $\E\qua{\abs{f} + \abs{F}_2}$, we conclude that $F$ is integrable. However, it is not clear if such a conclusion holds true for $f \in L\log^{1/2} L$: we proceed therefore generalizing some ideas from \cite{karaoco91}.
\end{remark}

\begin{proof}
If $f$ is smooth, then $F = \E\qua{\partial_sf | \mathcal{F}_s}$, so that $\abs{F_s} \le \E\qua{\abs{\partial_sf} | \mathcal{F}_s}$. By Fubini's theorem and Cauchy's inequality,
\[  \E\Big[\int_0^T \abs{F_s} ds\Big] \le \E\Big[\int_0^T \abs{\partial_sf}  ds \Big] \le T^{1/2} \E\qua{\abs{\nabla f}_2} \punto\]

Given a $BV$ function $f$, we consider a sequence of smooth functions $\pa{f_n}_{n\ge1}$, convergent to $f$ in $L^1\pa{\Prob}$, such that $\pa{\E\qua{\abs{\nabla f_n}_2}}_{n\ge1}$ converges to $\abs{Df} \pa{\Omega}$. Such a sequence exists by an application of Theorem \ref{theorem-ambrosio-fukushima} and a diagonal argument.

For every $n\ge1$, we write $F_n$ for the process that represents $f_n$, and $M^n$ for the $\mathcal{F}^\Prob$-martingale
\[ M^n_t = \E\qua{f_n | \mathcal{F}^\Prob_t} = \E\qua{f_n} + \int_0^t F^n_s dW_s \punto\]
We also write $M_t = \E\qua{f | \mathcal{F}^\Prob_t}$, and $\Delta M^*_n = \sup_{0\le t \le T} \abs{M^n_t - M_t}$. By Doob's maximal inequality, we have
\[ \Prob\pa{ \Delta M^*_n > \epsilon } \le \frac 1 \epsilon \E\qua{ \abs{M^n_T - M_T} }\virgola \]
for every $\epsilon > 0$. By the Burkholder-Gundy good-$\lambda$-inequality (see \cite{burkholder-1973}, and \cite{karaoco91}),
\[ \Prob\pa{ \int_0^T \abs{F^n_s - F_s}^2 ds > 4 \lambda^2, \Delta M^*_n \le \delta \lambda}\le \delta^2 \Prob\pa{  \int_0^T \abs{F^n_s - F_s}^2 ds > \lambda^2 }\virgola\]
for every $\lambda >0$ and $\delta \in \pa{0,1}$. This leads immediately to the inequality
\[ \Prob\pa {\int_0^T \abs{F^n_s - F_s}^2 ds > 4\lambda^2 } \le \delta^2 +\frac{1}{\delta \lambda} \E\qua{ \abs{M^n_T - M_T} } \virgola\]
which implies that
\[ \int_0^T \abs{F^n_s - F_s}^2ds \to 0 \]
in probability, and without loss of generality, $\Prob$ almost surely. By Cauchy's inequality, $\int\abs{F^n_s}ds$ converges to $\int\abs{F_s}ds$, $\Prob$ almost surely so, by Fatou's lemma, we conclude that
\[\E\Big[\int_0^T \abs{F_s} ds \Big]  \le T^{1/2} \liminf_n \E\qua{\abs{\nabla f_n}_2} =  T^{1/2} \abs{Df}\pa{\Omega}  \punto \]
\end{proof}

\begin{proof} (General case.) We can argue as in the $L^2\pa{\Prob}$ case, since by the lemma above, $F.\pa{\Prob \otimes \lambda}$ defines a real measure on the product space. It is therefore sufficient to show that the integrals of smooth processes, taken with respect to $Df$ and $F.\Prob \otimes \lambda$, coincide. We take $0< s < t \le T$ and a smooth $g$, measurable with respect to $\mathcal{F}_s$.

Let $\pa{\tau_n}_{n\ge1}$ be a sequence of predictable stopping times increasing to $T$, such that $F^n = I_{[0,\tau_n[} F \in \mathcal{M}^2$. Then,
\[ f_n = \E\qua{f | \mathcal{F}^\Prob_{\tau_n} } = \int_0^T F^n_r \,dW_r \virgola\]
so that $\pa{f_n}_{n\ge1}$ an $\pa{\mathcal{F}^\Prob_{\tau_n}}_{n\ge1}$-martingale, closed by $f$, in $L\log^{1/2}L$. By Lemma \ref{lemma-martingale-llog}, $\pa{f_n}_{n\ge1}$ converges to $f$ in $L\log^{1/2}L$.

For every $n\ge1$, by the Wiener-It\^o isometry,
\begin{equation}\label{equation-identity-ito-clark} \E\qua{ f_n \,g \pa{W_t - W_s} }  = \E\qua{ \int _0^T g I_{]s,t]}(r)\, F^n_r\,dr }\punto\end{equation}

By inequality \eqref{equation-holder-orlicz}, as $n\to \infty$, the left member converges to
\[ \E\qua{f g \pa{W_t - W_s} } = \int g I_{]s,t]} \, d Df \punto\]
For the member on the right, we use Lebesgue's dominated convergence theorem, since $g F \in L^1\pa{\Omega\times [0,T],\Prob\otimes \lambda}$. Therefore, identity \eqref{equation-identity-ito-clark} holds true with $f$ in place of $f_n$, and we conclude as we did for the  $L^2\pa{\Prob}$ case.\end{proof}

\section{Applications}\label{section-5}

We conclude by showing how the extended Ocone-Karatzas formula, together with the chain rule, allows us to find explicit representations. We consider first the case of $f= \phi\pa{W\pa{k}}$ and then the case of $f = \phi\pa{M}$ where $M = \sup_{s\le T} W_s$, with $\phi$ as introduced in Section \ref{section-3}.

\subsection{The case of cylindrical random variables}

Given $k \in L^2\pa{0,T}$, with $\abs{k}_2 = 1$, we choose a representative such that $k\pa{s} = 0$ if $\abs{I_{]s,T]}k}_2 = 0$, for every $s \in [0,T]$, and write
\[K\pa{ x, s } = k\pa{s} \frac{\exp\gra{-x^2/(2 \abs{I_{]s,T]}k}^2_2)} }{\sqrt{2 \pi} \abs{I_{]s,T]}k}_2 } \virgola\]
where we define $K\pa{x,s} = 0$ if $k\pa{s} = 0$.

\begin{proposition}\label{prop-app-1}
Let $f = \phi \pa{ W\pa{k} }$, with the notation as above. Then
\[ f  = \E\qua{f} + \int_0^T \qua{ \int_\erre K\pa{ x-W\pa{I_{]0,s]}k}, s }  D\phi\pa{dx}} dW_s \punto \]
\end{proposition}

\begin{proof}
It is sufficient to prove that, for all $s \in [0,T]$ and every smooth function $u$, measurable with respect to $\mathcal{F}_s$, then
\[ \frac{\exp\gra{-x^2/2}}{\sqrt{2\pi} } \E\qua{\, u k\pa{s} \, | W\pa{k} = x } = \E\qua{\, u\, K\pa{ x-W\pa{I_{]0,s]}k}, s } \, }\punto \]
Indeed, we define $U\pa{ r, x }$ as the right member above if $r\ge s$ and $U\pa{ r, x } = 0$ for $r<s$. Since the right member above is a continuous function of $x$ and the inequality
\[  \E\qua{\, u k\pa{s} \, | W\pa{k} = x } \le \norm{u} \abs{k\pa{s}} \]
holds true for all $s \in [0,T]$, we apply Theorem \ref{theorem-chain-rule-4-disintegrated-useful}, Fubini's theorem and we find that
\[ \E\qua{\, u\, \int_\erre K\pa{ x-W\pa{I_{]0,s]}k}, s } D\phi\pa{dx}\, }  = \int u D_s f \punto \]
for $\lambda$ a.e. $s \in [0,T]$ and every smooth function $u$ as above. By Remark \ref{remark-H-density}, we conclude.

Therefore, given a continuous function $\psi$, with compact support, we have to show that
\[ \E\qua{ \psi\pa{W\pa{k}} u k\pa{s}} = \int_\erre \psi \pa{x} \E\qua{\, u \,  K\pa{ x-W\pa{I_{]0,s]}k}, s }\, } dx\punto \]
We take the conditional expectation with respect to $\mathcal{F}_s$,
\[\E\qua{ \psi \pa{W\pa{k}} u k\pa{s}} = 
\E\qua{ u \int_\erre \psi\pa{ W\pa{kI_{]0,s]}} + y} K\pa{ y, s } dy } \punto\]
We conclude thanks to the change of variables $x = W\pa{kI_{]0,s]}} + y$ and exchanging integration and expectation.
\end{proof}

\begin{example}
We consider $f = I_{\gra{W_t \ge 0}}$, for some $ 0< t \le T$. Since $DI_{\gra{\sqrt{t} x \ge0}} = \sqrt{t} \delta_0$, we find the representation
\begin{equation} \label{equation-clark-brownian-bridge} I_{\gra{W_t \ge 0}} = \frac 1 2 + \int_0^t \frac{\exp\gra{-W_s^2/2\pa{t-s}} }{\sqrt{2 \pi \pa{t-s}} } dW_s \punto \end{equation}
\end{example}

\begin{remark}
Formula \eqref{equation-clark-brownian-bridge} above is well known (e.g.\ \cite{üstünel1995introduction}, p.\ 68) and there exists even an alternative approach (see \cite{springerlink:10.1023/A:1011259820029}), which produces similar formulas, valid for bounded cylindrical random variables.

Indeed, the following observation shows that  the two approaches are, in this case, equivalent. For every $s$, $K\pa{\cdot,s}\in C^1_b\pa{\erre}$, so we can integrate-by-parts on $\erre$ and obtain the representation
\[ f  = \E\qua{f} + \int_0^T \qua{ \int_\erre -\phi\pa{x} \partial_x K\pa{ x-W\pa{I_{]0,s]}k}, s }  dx } dW_s \virgola \]
which is a special case of Theorem 3.4 in \cite{springerlink:10.1023/A:1011259820029}.

We remark that, in the next application, such an integration-by-parts will fail in general, since the correspondent $K$ is less regular.
\end{remark}

\subsection{The case of the maximum of the Wiener process}

Given $0\le s< t \le T$, we write $M_{[s,t]} = \sup_{s\le r \le t} \pa{W_r - W_s}$, so that,
\begin{equation}\label{equation-maximum} M = M_{[0,T]} = M_{[0,s]} \lor \pa{W_s + M_{[s,T]}} \end{equation}
and $M_{[s,T]}$ is independent of $\mathcal{F}_s$, with a.c.\ law (with respect to $\lambda$), of density
\[ m_{T-s} \pa{x} = 2  \frac{\exp\gra{-x^2/2\pa{T-s}} }{\sqrt{2 \pi \pa{T-s}} } I_{\gra{x>0}} \punto\]

\begin{proposition}\label{prop-app-2}
Let $f = \phi \pa{M}$, with the notation as above. Then
\begin{equation} f  = \E\qua{f} +  \int_0^T \qua{\int_\erre m_{T-s}\pa{x-W_s} I_{\gra{x > M_{[0,s]}}} D\phi\pa{dx}}dW_s \punto \end{equation}
\end{proposition}

\begin{remark}\label{remark-derivative-maximum}
It is known (see e.g.\ \cite{nualart1995malliavin}, Proposition 2.1.10, p.\ 109) that $M$ admits a Malliavin derivative, $\nabla M = I_{[0,\sigma[}$, where 
\[\sigma = \inf\gra{ t: W_{t} \ge W_s, \,\, \forall \, 0 \le s \le T } \punto\]
We note that $\partial_s M$ is the indicator function of $\gra{  M_{[0,s]} < W_s + M_{[s,T]}}$.
\end{remark}

\begin{proof}
It is sufficient to prove that, for all $s \in ]0,T[$ and every smooth function $u$, measurable with respect to $\mathcal{F}_s$, then
\[ m_T\pa{x} \E\qua{\, u \partial_s M \, | M = x } = \E\qua{\, u\, m_{T-s}\pa{x-W_s} I_{\gra{x > M_{[0,s]}}} \, }\punto \]
Indeed, we define $U\pa{ r, x }$ as the right member above if $r\ge s$ and $U\pa{ r, x } = 0$ for $r<s$. Since the right member above is a continuous function of $x$ and the inequality
\[  \E\qua{\, u \partial_s M \ \, | M = x } \le \norm{u}  \]
holds true for all $s \in ]0,T[$, we conclude as we did for cylindrical functions.	

Therefore, given a continuous function $\psi$, with compact support, we have to show that
\[ \E\qua{ \psi\pa{M} u \partial_s M } = \int_\erre \psi \pa{x} \E\qua{\, u \, m_{T-s}\pa{x-W_s} I_{\gra{x > M_{[0,s]}}} } dx \punto \]
We take the conditional expectation with respect to $\mathcal{F}_s$ and use the expression for $M$ in \eqref{equation-maximum} $\partial_s M$ in Remark \ref{remark-derivative-maximum},
\[\E\qua{ \psi \pa{M} u \partial_s M} = \E\qua{ u \int_\erre \psi\pa{ W_s + y } I_{ \gra{W_s + y  > M_{[0,s]}} } m_{T-s}\pa{y} dy } \punto\]
With a change of variables $x = W_s + y$ and exchanging integration and expectation, we conclude.
\end{proof}

\begin{example}
We consider the function $I_{\gra{M \ge y}}$, for some $y>0$. Since $DI_{\gra{ x \ge y}} = \delta_y$, we find the representation
\[I_{\gra{M \ge y}} = \Prob\pa{M \ge y} + \int_0^{\tau_y} m_{T-s}\pa{y-W_s} dW_s \virgola \]
where $\tau_y = \inf\gra{ 0\le s\le T : W_s = y }$ is the time of the first visit at $y$.
\end{example}

\section*{Acknowledgements}

The authors thank L.\ Ambrosio for many helpful discussions and suggestions.

\bibliography{bvclark}

\end{document}